\documentclass[reqno,11pt,a4paper]{amsart}
\usepackage{amsmath,amssymb,latexsym,hyperref}
\usepackage{amscd}

\usepackage{enumerate}
\usepackage[all]{xy}
\usepackage{fancyhdr}
\usepackage{color}
\usepackage{comment} 

\numberwithin{equation}{section}

\usepackage[english]{babel}

\newcommand{\p}{\mathfrak p}

\newcommand{\loc}{\mathrm{loc}}
\newcommand{\Char}{\mathrm{Char}}

\newcommand{\col}{\mathrm{Col}}
\newcommand{\ucol}{\widetilde{\col}}
\newcommand{\QQ}{\mathbb{Q}}

\newcommand{\Qp}{\mathbb{Q}_p}
\newcommand{\Zp}{\mathbb{Z}_p}

\newcommand{\zk}{z_\mathrm{Kato}}

\newcommand{\GL}{\mathrm{GL}}
\newcommand{\Hom}{\mathrm{Hom}}

\newcommand{\Sel}{\mathrm{Sel}}

\newcommand{\cyc}{\mathrm{cyc}}

\newcommand{\tor}{\mathrm{tor}}

\newcommand{\rank}{\mathrm{rank}}
\newcommand{\Ep}{E_{p^\infty}}
\newcommand{\HIw}{H^1_{\mathrm{Iw}}}

\newcommand{\Gal}{\mathrm{Gal}}
\newcommand{\cZ}{\mathcal{Z}}

\usepackage[OT2,OT1]{fontenc}
\newcommand\cyr{%
\renewcommand\rmdefault{wncyr}%
\renewcommand\sfdefault{wncyss}%
\renewcommand\encodingdefault{OT2}%
\normalfont\selectfont}

\DeclareTextFontCommand{\textcyr}{\cyr}

\newtheorem{theorem}{Theorem}[section]
\newtheorem{proposition}[theorem]{Proposition}
\newtheorem{lemma}[theorem]{Lemma}

\newtheorem{question}[theorem]{Question}

\theoremstyle{remark}

\definecolor{Green}{rgb}{0.0, 0.5, 0.0}

\newcommand{\Keywords}[1]{\par\noindent
{\small{Keywords and phrases}: #1}}

\newcommand{\AMS}[1]{\par\noindent
{\small{AMS Subject Classification}: #1}}

\author[A. Lei]{Antonio Lei}
\address{(A. Lei) D\'epartement de Math\'ematiques et de Statistiques, Universit\'e Laval, Pavillon Alexandre-Vachon, 1045 Avenue de la M\'edecine, Qu\'ebec, QC, Canada G1V 0A6}
\email{antonio.lei@mat.ulaval.ca}

\author[R. Sujatha]{R. Sujatha}
\address{(R. Sujatha) Mathematics Department, 1984, Mathematics Road, University of British Columbia,  Vancouver, 
Canada V6T 1Z2}
\email{sujatha@math.ubc.ca}

\begin{document}

\title{On fine Selmer groups and the greatest common divisor of signed and chromatic $p$-adic $L$-functions}

\begin{abstract}
Let $E/\mathbb{Q}$ be an elliptic curve and $p$ an odd prime where $E$ has good supersingular reduction. Let $F_1$ denote the characteristic power series of the Pontryagin dual of the fine Selmer group of $E$ over the cyclotomic $\mathbb{Z}_p$-extension of $\mathbb{Q}$ and let $F_2$  denote the greatest common divisor of Pollack's plus and minus $p$-adic $L$-functions or Sprung's sharp and flat $p$-adic $L$-functions attached to $E$, depending on whether   $a_p(E)=0$ or $a_p(E)\ne0$. We study a link between the divisors of $F_1$ and $F_2$ in the Iwasawa algebra.   This gives new insights into problems posed by Greenberg and Pollack--Kurihara on these elements.  
\end{abstract}

\maketitle
{
\AMS{11R23, 11G05}
}

\section{Introduction}\label{intro}
Let $p$ a fixed odd prime number. We write  $\QQ_\cyc$ for the cyclotomic $\Zp$-extension of $\QQ$ and let  $\Gamma$ denote the Galois group $\Gal(\QQ_\cyc/\QQ)$. The Iwasawa algebra $\Lambda=\Zp[[\Gamma]]$ is defined to be $\varprojlim \Zp[\Gamma/\Gamma^{p^n}]$, where the connecting maps are projections. After fixing a topological generator $\gamma$ of $\Gamma$, there is an isomorphism of rings $\Lambda\cong\Zp[[X]]$, sending $\gamma$ to $X+1$. 

 Let $E/\QQ$ be an elliptic curve and write $\Sel_0(E/\QQ_\cyc)$ for the fine Selmer group of $E$ over $\QQ_\cyc$ as defined in \cite{CS} (whose precise definition is reviewed in \eqref{eq:fine} below). It has been shown by Kato in \cite{kato} that $\Sel_0(E/\QQ_\cyc)^\vee$ is a finitely generated $\Lambda$-torsion module.  Conjecture A in \cite{CS} further predicts that it should be a finitely generated $\Zp$-module, which is equivalent to saying that its $\mu$-invariant is zero. Examples validating this conjecture can be found in a recent work of Kundu and the second named author \cite{KS}.

For an integer $n\ge1$, we write $$\Phi_n=\frac{(1+X)^{p^n}-1}{(1+X)^{p^{n-1}}-1}\in\Lambda$$ for the $p^n$-th cyclotomic polynomial in $1+X$. Let $K_n$ denote the unique sub-extension of $\QQ_\cyc$ such that $[K_n:\QQ]=p^n$.  Define 
\[
e_n=\frac{\rank E(K_n)-\rank E(K_{n-1})}{p^{n-1}(p-1)}.
\]
When $n=0$, we define $\Phi_0=X$ and $e_0=\rank E(\QQ)$.
We recall from \cite[Problem~0.7]{KP} that the following problem was posed by Greenberg:
\begin{equation}\label{Gr}\tag{Gr}
\Char_\Lambda \Sel_0(E/\QQ_\cyc)^\vee\stackrel{?}{=}\Big(\prod_{{e_n\ge 1},{n\ge0}}\Phi_n^{e_n-1}\Big).
\end{equation}
Here, $\Char_\Lambda M$ denotes the $\Lambda$-characteristic ideal of a finitely generated torsion $\Lambda$-module $M$. In particular, if \eqref{Gr} holds, then Conjecture A of \cite{CS} holds as well.

We now turn our attention to  the case where $E$ has supersingular reduction at $p$.  The classical $p$-adic $L$-functions attached to $E$ are $p$-adic power series with unbounded denominators (in particular, they are not elements of $\Lambda$).  In \cite{pollack}, {under the hypothesis that $a_p(E)=0$,} Pollack introduced the so-called plus and minus $p$-adic $L$-functions $L_p^+(E)$ and $L_p^-(E)$, which are non-zero elements of $\Lambda$, interpolating complex $L$-values of $E$ twisted by Dirichlet characters factoring through $\Gamma$. They can be regarded as the analytic analogues of  certain cotorsion Selmer groups $\Sel^\pm(E/\QQ_\cyc)$ defined by Kobayashi \cite{kobayashi}. In fact, Kobayashi formulated the following main conjecture
\begin{equation}
    \Char_\Lambda \Sel^\pm(E/\QQ_\cyc)^\vee\stackrel{?}{=}\left(L_p^\pm(E)\right).
\label{Ko}\tag{Ko}
\end{equation}
{When $a_p(E)\ne0$, Sprung \cite{sprung} generalized the works of Pollack and Kobayashi by introducing the sharp and flat $p$-adic $L$-functions $L_p^\sharp(E)$ and $L_p^\flat(E)$ as well as the corresponding Selmer groups $\Sel^{\sharp}(E/\QQ_\cyc)$ and $\Sel^{\flat}(E/\QQ_\cyc)$. He showed that there exists $\star\in \{\sharp,\flat\}$ such that  $L_p^\star(E)\ne 0$ and that $\Sel^{\star}(E/\QQ_\cyc)^\vee$ is $\Lambda$-torsion (see Proposition~6.14 and Theorem~7.14 of \cite{sprung}). Moreover, he formulated the analogue of \eqref{Ko}: If $\star\in \{\sharp,\flat\}$ is such that $L_p^\star(E)\ne 0$, then
\begin{equation}
    \Char_\Lambda \Sel^{\star}(E/\QQ_\cyc)^\vee\stackrel{?}{=}\left(L_p^{\star}(E)\right).
\label{Sp}\tag{Sp}
\end{equation}
Furthermore, by \cite[Theorem~7.4]{kobayashi} and \cite[discussion on P.1505]{sprung} respectively, \eqref{Ko} and \eqref{Sp} are equivalent to Kato's main conjecture in \cite{kato},} which can be expressed as
\begin{equation}\label{Ka}\tag{Ka}
\Char_\Lambda \Sel_0(E/\QQ_\cyc)^\vee\stackrel{?}{=}\Char_\Lambda\HIw(\QQ,T)/\cZ,
\end{equation}
where $T$ is the $p$-adic Tate module of $E$, $\HIw(\QQ,T)$ is the inverse limit of certain global Galois cohomological groups over $K_n$ and $\cZ$ is the $\Lambda$-module generated by certain zeta elements (we will review the definition of these objects in the main part of the article).

Kato showed that there exists an integer $n$ such that 
\[
\Char_\Lambda \Sel_0(E/\QQ_\cyc)^\vee\supset p^n\Char_\Lambda\HIw(\QQ,T)/\cZ
\]
using the theory of Euler systems. Furthermore, if the Galois representation $G_\QQ\rightarrow \GL_{\Zp}(T)$  is surjective, then we may take $n=0$. In other words, the inclusion $\supset$ holds in \eqref{Ka}. This has the consequence that {one inclusion in the main conjectures \eqref{Ko} and \eqref{Sp} also holds, namely:
\begin{equation}\label{eq:1sideIMC}
\left(L_p^\star(E)\right)\subset \Char_{\Lambda}\Sel^\star(E/\QQ_\cyc)^\vee,
\end{equation}
where $\star\in\{+,-\}$ or $\{\sharp,\flat\}$, depending on whether $a_p(E)=0$ or $a_p(E)\ne0$.}

When  $E$ has complex multiplication {(in which case $a_p(E)$ is always $0$)},  Pollack and Rubin \cite{PR} showed that  \eqref{Ko} holds. Consequently, \eqref{Ka} holds as well. In the non-CM case, recent progress on these conjectures has been made by Wan \cite{wan} {and Sprung \cite{sprung2}}.

It follows from their definitions that $\Sel_0(E/\QQ_\cyc)$ is a subgroup of  both plus and minus {(or sharp and flat)} Selmer groups. In particular, we have the inclusions
\begin{equation}
 \Char_{\Lambda}\Sel^\star(E/\QQ_\cyc)^\vee\subset\Char_{\Lambda}\Sel_0(E/\QQ_\cyc)^\vee .
\label{eq:finepm}
\end{equation}
Pollack has written a MAGMA program which computes numerically  the $p$-adic $L$-functions $L_p^\star(E)$ for a given $E$ (see \url{http://math.bu.edu/people/rpollack/Data/data.html})\footnote{Even though Pollack's algorithm was written before Sprung's $p$-adic $L$-functions $L_p^{\sharp/\flat}(E)$ were defined, it in fact computes the Iwasawa invariants considered by Perrin-Riou in \cite{P-R} when $a_3(E)=\pm 3$. These in turn give the Iwasawa invariants of $L_p^{\sharp/\flat}(E)$ (see \cite[\S5]{sprung-crelles}). We thank Robert Pollack and Florian Sprung for explaining this to us.}. One can observe that the $\mu$-invariants of $L_p^\star(E)$ turn out to be zero in all the examples that have been considered by Pollack. Therefore, on combining \eqref{eq:1sideIMC} and \eqref{eq:finepm},  we deduce that the $\mu$-invariant of  $\Char_{\Lambda}\Sel_0(E/\QQ_\cyc)^\vee$ is also zero.  In particular, this gives evidence towards the validity of Conjecture~A in \cite{CS}. In this article, we are interested in the following question:
\begin{question}\label{Q}
What can \cite[Conjecture A]{CS} tell us about the $\mu$-invariants of  $L_p^\pm(E)$ and $L_p^{\sharp/\flat}(E)$?
\end{question}

In \cite[Problem~3.2]{KP}, {under the assumption that $a_p(E)=0$,} the following problem was posed by Kurihara and Pollack:
\begin{equation}\label{KP}\tag{KP}
\gcd\left(L_p^+(E),L_p^-(E)\right)\stackrel{?}{=}X^{e_0}\prod_{{e_n\ge 1},{n\ge1}}\Phi_n^{e_n-1}.
\end{equation}
Here, we express the greatest common divisor of two elements in $\Lambda$ in the form $p^\mu h\in\Lambda$, where $h$ is a distinguished polynomial.
The two problems \eqref{Gr} and \eqref{KP} are intimately linked. Indeed, Kurihara and Pollack showed in \cite[\S3]{KP} that under certain hypotheses, they are equivalent to each other. Furthermore, they have found several  numerical examples where the answers to both \eqref{Gr} and \eqref{KP} are affirmative.
We observe that the problems \eqref{Gr} and \eqref{KP} suggest that the following equality holds
\begin{equation}
\left(\gcd\left(L_p^+(E),L_p^-(E)\right)\right)\stackrel{?}{=}X^{\delta_E}\Char_{\Lambda}\Sel_0(E/\QQ_\cyc)^\vee,
\label{eq:Xgcd}    
\end{equation}
where $\delta_E\in\{0,1\}$. The appearance of the term $X^\delta$ on the right-hand side originates from the discrepancy\footnote{This discrepancy seems to be related to the fact that 
\[
E^+(\QQ_{\cyc,\p})\bigcap E^-(\QQ_{\cyc,\p})=E(\Qp),
\]
where $E^\pm(\QQ_{\cyc,\p})$ are Kobayashi's plus and minus norm groups which are used to define  $\Sel^\pm(E/\QQ_\cyc)$ (see \S\ref{S:selmer} for more details). This suggests that the plus and minus Selmer groups might capture more information on the Mordell-Weil group $E(\QQ)$ than the fine Selmer group.} of the exponents of $X$ in \eqref{Gr} and \eqref{KP}.

The main result of the present article is the following theorem which gives evidence towards \eqref{eq:Xgcd}. It can also be regarded as partial evidence towards \eqref{Gr} and \eqref{KP}.

\begin{theorem}\label{thm:main}
Suppose that $E/\QQ$ is an elliptic curve with supersingular reduction at $p$.  {In the case where $a_p(E)\ne0$, we assume that both $L_p^\sharp(E)$ aand $L_p^\flat(E)$ are non-zero.} Furthermore, assume that \eqref{Ka} holds (in particular, \eqref{Ko} and {\eqref{Sp} also hold}). Let $f\in \Lambda$ be an irreducible element that is coprime to $X$. If $a_p(E)=0$, then $f$ divides $\gcd\left(L_p^+(E),L_p^-(E)\right)$ if and only if $f$ divides $\Char_{\Lambda}\Sel_0(E/\QQ_\cyc)^\vee$. {Likewise, if $a_p(E)\ne0$, then $f$ divides $\gcd\left(L_p^\sharp(E),L_p^\flat(E)\right)$ if and only if $f$ divides $\Char_{\Lambda}\Sel_0(E/\QQ_\cyc)^\vee$. }
\end{theorem}

In particular, Theorem~\ref{thm:main} gives the following answer to Question~\ref{Q} on taking $f=p$: If \eqref{Ka} holds, then \cite[Conjecture~A]{CS} (which says that $p$ does not divide $\Char_{\Lambda}\Sel_0(E/\QQ_\cyc)^\vee$) is equivalent to $p\nmid \gcd\left(L_p^+(E),L_p^-(E)\right)$ {(or $p\nmid \gcd\left(L_p^\sharp(E),L_p^\flat(E)\right)$)}, which is the same as saying that the $\mu$-invariant of at least one of the two $p$-adic $L$-functions $L_p^\pm(E)$ {(or $L_p^{\sharp/\flat}(E)$)} is zero.

\subsection*{Outline of the article} We review general definitions and preliminary results on Iwasawa modules and Selmer groups in Section~\ref{sec:notation}. The proof of Theorem~\ref{thm:main} is then given  in Section~\ref{S:proof}.   At the end of the article (Section \ref{S:ex}), we discuss some numerical examples.

\subsection*{Acknowledgment}
This work was initiated during the first named author's visit to PIMS in March 2020. He would like to thank PIMS for the hospitality. We would like to thank Filippo Nuccio, Robert Pollack, Florian Sprung and Chris Wuthrich for helpful discussions during the preparation of this article.
Both authors also gratefully acknowledge support of their respective
NSERC  Discovery Grants. {Finally, we thank the anonymous referee for very helpful comments and suggestions that led to many improvements to the presentation of the article.}

\section{Notation and preliminary results}\label{sec:notation}

\subsection{Generalities on $\Lambda$-modules}
Let $M$ be a $\Lambda$-module, we write $M^\vee$ for its Pontryagin dual
\[
\Hom_{\mathrm{conts}}(M,\Qp/\Zp).
\]
 We let $M_\tor$ denote  the maximal torsion submodule of $M$. 

If $M$ is a finitely generated $\Lambda$-module,  there is a pseudo-isomorphism of $\Lambda$-modules
\begin{equation}
M\sim \Lambda^{ r}\oplus \bigoplus_{i=1}^s \Lambda/p^{a_i}\oplus \bigoplus_{i=1}^t\Lambda/(F_i^{b_i}),
\label{eq:pseudo}
\end{equation}
where $r,s,t,a_i,b_i$ are non-negative integers and $F_i$ are irreducible distinguished polynomials. We define the $\mu$- and $\lambda$-invariants of $M$ by
\begin{align*}
\mu(M)&:=\sum_{i=1}^sa_i,\\
\lambda(M)&:=\sum_{i=1}^t\deg(F_i).
\end{align*}
In the case where $M$ is $\Lambda$-torsion, we have $r=0$ and the characteristic ideal of $M$ is defined to be
\[
\Char_\Lambda(M)=\left(p^{\mu(M)}\prod_{i=1}^t F_i^{b_i}\right)\Lambda.
\]

Given any irreducible element $f\in\Lambda$, we write  $M[f]$ for the $\Lambda$-submodule of $M$ defined by
\[
\{m\in M:f\cdot m=0\}.
\]
The following lemmas will be employed in our proof of Theorem~\ref{thm:main}.

\begin{lemma}\label{lem:coprime}
Let $M$ be a finitely generated $\Lambda$-module and $f\in\Lambda$ an irreducible element. Then, $M[f]$ is finite if and only if $f\nmid\Char_{\Lambda}(M_\tor)$.
\end{lemma}
\begin{proof}
We can see from the the pseudo-isomorphism \eqref{eq:pseudo} that $M[f]$ is finite if and only if $\left(\Lambda/(g^n)\right)[f]$ is finite for all $\Lambda/(g^n)$ that appear on the right-hand side of \eqref{eq:pseudo}. Therefore, without loss of generality, we may assume that $M=\Lambda/(g^n)$ for some irreducible element $g$  of $\Lambda$ and $n\ge1$. 
It is  clear that $M[f]=0$ if $\gcd(f,g)=1$ and $M[f]=f^{n-1}\Lambda/(f^n)$ if $f=g$, which is of infinite cardinality. Therefore, our lemma follows.
\end{proof}

\begin{lemma}\label{lemmas}
Let 
\[
0\rightarrow A\rightarrow B\rightarrow C\rightarrow0
\]
be a short exact sequence of finitely generated $\Lambda$-modules  and $f\in\Lambda$ an irreducible element.
\begin{enumerate}[(a)]
    \item Suppose that $f\nmid \Char_\Lambda B_\tor$, then $f\nmid \Char_\Lambda A_\tor$.
    \item Suppose that $f\nmid \Char_\Lambda A_\tor$ and $f\nmid\Char_\Lambda C_\tor$, then $f\nmid \Char_\Lambda B_\tor$.
    \item Suppose that $A$ is a torsion $\Lambda$-module, then we have the equality
    \[
    \Char_\Lambda(A)\Char_\Lambda(C_\tor)=\Char_\Lambda (B_\tor).
    \]
\end{enumerate}
\end{lemma}
\begin{proof}
Part (a) follows from Lemma~\ref{lem:coprime} and the fact that  $A[f]$ injects into $B[f]$. Part (b) follows from the exact sequence
\[
0\rightarrow A[f]\rightarrow B[f]\rightarrow C[f]\rightarrow \cdots
\]
and Lemma~\ref{lem:coprime}. Finally, part (c) follows from \cite[proof of Proposition~2.1]{HL}.
\end{proof}

\subsection{Galois cohomology and Selmer groups}\label{S:selmer}
As in the introduction, $T$ denotes the $p$-adic Tate module of $E$. Let $S$ be a finite set of primes of $\QQ$ including the bad primes of $E$, the prime $p$ and the archimedean prime. Given a finite extension of $F$, we write $G_{F,S}$ for the Galois group of the maximal extension of $F$ that is unramified outside $S$. We define
\[
\HIw(\QQ,T)=\varprojlim H^1(G_{K_n,S},T),
\]
where $K_n$ is the unique sub-extension of $\QQ_{\cyc}$ such that $[K_n:\QQ]=p^n$ and the connecting maps in the inverse limit are given by corestrictions. It has been shown in \cite{kato} that there exists a so-called zeta element $\zk\in\HIw(\QQ,T)$, which, when localized at $p$, interpolates complex $L$-values of $E$ twisted by Dirichlet characters factoring through $\Gamma$ under  the dual exponential map of Bloch-Kato as defined in \cite{BK}. We define $\cZ\subset\HIw(\QQ,T)$ to be the $\Lambda$-submodule generated by $\zk$.

Locally, we define 
\[
\HIw(\Qp,T)=\varprojlim H^1(K_{n,v_n},T),
\]
where $v_n$ denotes the unique prime of $K_n$ lying above $p$ and the connecting maps in the inverse limit are again given by corestrictions. {The restriction maps $H^1(K_n,T)\rightarrow H^1(K_{n,v_n},T)$ give}
\[
\loc_p:\HIw(\QQ,T)\hookrightarrow \HIw(\Qp,T),
\]
where the injectivity follows from \cite[Theorem~7.3]{kobayashi} {and \cite[(3) on P.1504]{sprung}}.
Let us write $\cZ_\loc$ for the image of $\cZ$ under the localization map, that is
\[
\cZ_\loc=\loc_p(\cZ).
\]

We finish this section by defining the various Selmer groups of $E$ over $\QQ_\cyc$ studied in this article. {Recall that if $K$ is a number field, the classical $p$-primary Selmer group of $E$ over $K$ is given by
\[
\Sel_{p^\infty}(E/K)= \ker\left(H^1(G_{K,S},\Ep)\rightarrow \bigoplus_{v\in S}J_v(K)\right),
\]
where $J_v(K)$ is defined to be 
$\displaystyle
 \bigoplus_{w|v}\frac{H^1(K_{w},\Ep)}{E(K_w)\otimes\Qp/\Zp}
$
(here, the direct sum runs over all places of $K$ above $v$). 
The classical $p$-primary Selmer group of $E$ over $\QQ_\cyc$ is given by
\[
\Sel_{p^\infty}(E/\QQ_\cyc)= \varinjlim_n \Sel_{p^\infty}(E/K_n),
\]
where the connecting maps are given by restrictions.}

The fine Selmer group of $E$ over $\QQ_\cyc$ is given by
\begin{equation}
   \Sel_0(E/\QQ_\cyc)= \ker\left(\Sel_{p^\infty}(E/\QQ_\cyc)\rightarrow H^1(\QQ_{\cyc,\p},\Ep)\right),
 \label{eq:fine}
\end{equation}
where $\p$ denotes the unique prime of $\QQ_\cyc$ above $p$ {(see \cite[(58) on P.828]{CS})}. {When $a_p(E)=0$,} Kobayashi's plus and minus Selmer groups are defined by
\[
\Sel^\pm(E/\QQ_\cyc)= \ker\left(\Sel_{p^\infty}(E/\QQ_\cyc)\rightarrow \frac{H^1(\QQ_{\cyc,\p},\Ep)}{{E}^\pm(\QQ_{\cyc,\p})\otimes\Qp/\Zp}\right),
\]
where $E^\pm(\QQ_{\cyc,\p})$ are certain subgroups in $E$ defined by some ``jumping conditions" (see \cite[Definition~1.1]{kobayashi}). {When $a_p(E)\ne0$, Sprung's sharp and flat Selmer groups are defined by
\[
\Sel^{\sharp/\flat}(E/\QQ_\cyc)= \ker\left(\Sel_{p^\infty}(E/\QQ_\cyc)\rightarrow \frac{H^1(\QQ_{\cyc,\p},\Ep)}{{E}^{\sharp/\flat}(\QQ_{\cyc,\p})\otimes\Qp/\Zp}\right),
\]
where $E^{\sharp/\flat}(\QQ_{\cyc,\p})$ are given by the exact annihilators of certain Coleman maps (see \cite[Definition~7.9]{sprung}).
}

\section{Proof of Theorem~\ref{thm:main}}\label{S:proof}
{This section is dedicated to the proof of the main theorem of this article. We remark that some  of the ingredients of the proof were also utilized in \cite[proof of Proposition~3.4]{KP}.}

Throughout this section, $f\in \Lambda$ is a fixed irreducible element that is coprime to $X$. {We write $(\circ,\bullet)=(+,-)$ or $(\sharp,\flat)$, depending on whether $a_p(E)=0$ or $a_p(E)\ne0$.} Suppose that $f$ does not divide {$\gcd(L_p^\circ(E),L_p^\bullet(E))$}. Then it does not divide $\Char_\Lambda\Sel_0(E/\QQ_\cyc)^\vee$ by \eqref{eq:finepm}. This gives one of the two implications of Theorem~\ref{thm:main}. The rest of this section will be dedicated to the proof of the opposite implication, which is less straightforward. From now on, we assume that
\begin{equation}
    f\nmid\Char_\Lambda\Sel_0(E/\QQ_\cyc)^\vee.
\label{eq:hyp0}
\end{equation}
If we combine this with \eqref{Ka}, we have
\begin{equation}\label{eq:hyp}
   f\nmid \Char_\Lambda\HIw(\QQ,T)/\cZ. 
\end{equation}

The following proposition is one of the key ingredients of the proof of Theorem~\ref{thm:main}. 
\begin{proposition}\label{pro:key}
We have $$f\nmid \Char_\Lambda\left(\HIw(\Qp,T)/\cZ_\loc\right)_\tor.$$
\end{proposition}
\begin{proof}
The injectivity of $\loc_p$ gives the following short exact sequence
\[
0\rightarrow \HIw(\QQ,T)/\cZ\rightarrow \HIw(\Qp,T)/\cZ_\loc\rightarrow \HIw(\Qp,T)/\loc_p\left(\HIw(\QQ,T)\right)\rightarrow 0.
\]
Since the $\Lambda$-module $\HIw(\QQ,T)$ is of rank one (see \cite[Theorem~12.4]{kato}),  the first term of the short exact sequence is $\Lambda$-torsion. Therefore, thanks to Lemma~\ref{lemmas}(c), it is enough to show that
\[
f\nmid \Char_\Lambda\HIw(\QQ,T)/\cZ\quad \text{and}\quad f\nmid  \Char_\Lambda\left(\HIw(\Qp,T)/\loc_p\left(\HIw(\QQ,T)\right)\right)_\tor.
\]

The first indivisibility is a direct consequence of our hypothesis that both \eqref{Ka} and \eqref{eq:hyp0} hold. For the second indivisibility, we consider the Poitou-Tate exact sequence
\begin{equation}
    0\rightarrow \HIw(\QQ,T)\stackrel{\loc_p}{\longrightarrow}\HIw(\Qp,T)\rightarrow \Sel_{p^\infty}(E/\QQ_\cyc)^\vee\rightarrow  \Sel_0(E/\QQ_\cyc)^\vee\rightarrow0
\label{eq:PT}
\end{equation}
(which is obtained by taking inverse limit in \cite[(7.18)]{kobayashi}). By \cite[Corollary~2.5]{win}, we have the equality
\[
 \Char_\Lambda\Sel_{p^\infty}(E/\QQ_\cyc)^\vee_\tor=\Char_\Lambda\Sel_{0}(E/\QQ_\cyc)^\vee.
\]
Therefore, our hypothesis \eqref{eq:hyp0} tells us that 
\[
f\nmid  \Char_\Lambda\Sel_{p^\infty}(E/\QQ_\cyc)^\vee_\tor.
\]
We may therefore apply Lemma~\ref{lemmas}(a) to \eqref{eq:PT} to deduce that 
\[
f\nmid\Char_\Lambda\left(\HIw(\Qp,T)/\loc_p\left(\HIw(\QQ,T)\right)\right)_\tor
\]
as required.
\end{proof}

We recall from \cite[\S\S 8.5-8.6]{kobayashi} and {\cite[\S7]{sprung}} that there are two surjective $\Lambda$-homomorphisms 
\[
\col^{\circ/\bullet}:\HIw(\Qp,T)\rightarrow \Lambda,
\]
which are called the plus and minus {(or sharp and flat) Coleman maps of $E$. Furthermore, \cite[Theorem~6.3]{kobayashi} and \cite[Definition~6.1]{sprung} tell us that
\begin{equation}
    \label{eq:pm-Col-Lp}
    L_p^{\circ/\bullet}(E)=\col^{\circ/\bullet}\left(\loc_p \left(\zk\right)\right).
\end{equation}}
(Note that we are taking $\eta$ to be the trivial character in the notation of op. cit.)

If we write 
\begin{align*}
    \ucol:\HIw(\Qp,T)&\rightarrow \Lambda^{\oplus 2}\\
z&\mapsto  {\col^\circ(z)\oplus\col^\bullet(z)},
\end{align*}
then we have a short exact sequence of $\Lambda$-modules
\begin{equation}
    0\rightarrow \HIw(\Qp,T)\stackrel{\ucol}{\longrightarrow}\Lambda^{\oplus 2}\rightarrow \Zp\rightarrow 0
\label{eq:SES-KP}
\end{equation}
as given by \cite[Proposition~1.2]{KP} {and \cite[Proposition~4.7]{sprung}}.

If we combine \eqref{eq:pm-Col-Lp} and \eqref{eq:SES-KP}, we deduce the short exact sequence
\begin{equation}
    0\rightarrow \HIw(\Qp,T)/\cZ_\loc\rightarrow \Lambda^{\oplus 2}/(L_p^\circ(E)\oplus L_p^\bullet(E))\Lambda\rightarrow \Zp\rightarrow 0.
\end{equation}
But $\Char_\Lambda\Zp=(X)$, which is coprime to $f$ by assumption. Hence,  we deduce from Proposition~\ref{pro:key} and Lemma~\ref{lemmas}(b) that 
\[
f\nmid  \Char_\Lambda\left(\Lambda^{\oplus 2}/(L_p^\circ(E)\oplus L_p^\bullet(E))\Lambda\right)_\tor.
\]
In particular, $f\nmid \gcd(L_p^\circ(E),L_p^\bullet(E))$, which concludes the proof of Theorem~\ref{thm:main}.

\section{Numerical examples}\label{S:ex}

We discuss the two elliptic curves studied in \cite[\S10]{wuhtrich}, namely, $37A1$ and $53A1$, both of which are of rank one over $\QQ$ with $L(E/\QQ,1)=0$.

\subsection*{$E=37A1$}
According to \cite[Proposition~10.1]{wuhtrich}, the fine Selmer group of $E$ over $\QQ_\cyc$ is finite for all primes $p<1000$. In particular, 
\[
\Char_\Lambda \Sel_0(E/\QQ_\cyc)^\vee=\Lambda
\]
for these primes.
Note that $E$ has supersingular reduction at the primes $p=3,17,19 $ with $a_3(E)=-3$ and $a_{17}(E)=a_{19}(E)=0$. Theorem~\ref{thm:main} tells that if $f\in\Lambda$ is an irreducible element dividing $\gcd\left(L_p^\sharp(E),L_p^\flat(E)\right)$ (resp. $\gcd\left(L_p^+(E),L_p^-(E)\right)$) when $p=3$ (resp. $p=17$ or $19$), then $f$ has to be (up to a unit) equal to $X$. In fact, we can even work out the greatest common divisors explicitly in these cases.

Since $L(E/\QQ,1)=0$, it follows from the interpolation formulae of the $p$-adic $L$-functions given in \cite[Page~14980]{sprung} and \cite[Page~7]{kobayashi} that $X$ divides $L_p^{\sharp/\flat}(E)$ (when $p=3$) and $L_p^\pm(E)$ (when $p=17$ or $19$). According to Pollack's table  \url{http://math.bu.edu/people/rpollack/Data/37A.p} (see also \cite[Example~7.12]{sprung-ANT} where the case $p=3$ is discussed), one of the two $p$-adic $L$-functions has $\lambda$-invariant equal to 1. This implies  that
\[
\gcd\left(L_p^\sharp(E),L_p^\flat(E)\right)=X
\]
if $p=3$ and 
\[
\gcd\left(L_p^+(E),L_p^-(E)\right)=X
\]
if $p=17$ or $19$. Note in particular that the equation \eqref{eq:Xgcd} holds for this curve when $p=17$ and $19$. Furthermore, the $\mu$-invariants of $\Sel_0(E/\QQ_\cyc)^\vee$, $\Sel_{p^\infty}(E/\QQ_\cyc)_\tor^\vee$, $\Sel^{\sharp/\flat}(E/\QQ_\cyc)^\vee$ (when $p=3$) and $\Sel^\pm(E/\QQ_\cyc)^\vee$ (when $p=17,19$) are all zero.

\subsection*{$E=53A1$} Once again, $E$ is supersingular at $p=3$ with $a_3(E)=-3$. Wuthrich showed that the fine Selmer group over $\QQ_\cyc$ is finite when $p=3$ and Pollack's table \url{http://math.bu.edu/people/rpollack/Data/curves1-5000} tells us that $L_p^{\sharp/\flat}(E)=X$ (up to a unit). Therefore, we can deduce once more that
\begin{align*}
\Char_\Lambda \Sel_0(E/\QQ_\cyc)^\vee&=\Lambda,\\    
\gcd\left(L_p^\sharp(E),L_p^\flat(E)\right)&=X,
\end{align*}
illustrating Theorem~\ref{thm:main}.

Note that $E$ has supersingular reduction at $p=5,11$ and $a_5(E)=a_{11}(E)=0$. Pollack's table tells us that $L_p^\pm(E)=X$ (up to a unit) in these cases. We may apply the argument  in \cite[\S10]{wuhtrich} to deduce that the fine Selmer group of $E$ over $\QQ_\cyc$ is finite. This again illustrates Theorem~\ref{thm:main} and gives examples where the equality \eqref{eq:Xgcd} holds. Furthermore, the $\mu$-invariants of $\Sel_0(E/\QQ_\cyc)^\vee$, $\Sel_{p^\infty}(E/\QQ_\cyc)_\tor^\vee$, $\Sel^{\sharp/\flat}(E/\QQ_\cyc)^\vee$ (when $p=3$) and $\Sel^\pm(E/\QQ_\cyc)^\vee$ (when $p=5,11$) vanish.
\bibliographystyle{amsalpha}
\bibliography{references}

\end{document}